\documentclass[12pt]{amsart}
\usepackage{amsfonts}
\usepackage{amsmath}
\usepackage{amssymb,latexsym}
\usepackage[mathcal]{eucal}
\usepackage{amscd}
\usepackage{hyperref}

\input xy
\xyoption{all}

\newcommand{\ch}{\mathbf{1}}

\newcommand{\bX}{\mathbf{X}}

\newcommand{\Z}{\mathbb{Z}}

\newcommand{\C}{\mathbb{C}}
\newcommand{\N}{\mathbb{N}}
\newcommand{\T}{\mathbb{T}}

\newcommand{\Xb}{\mathbf{X}}

\newcommand{\al}{\alpha}

\newcommand{\del}{\delta}

\newcommand{\sig}{\sigma}

\newcommand{\la}{\lambda}

\newcommand{\ol}{\overline}

\newcommand{\br}{\vspace{3 mm}}

\newcommand{\tri}{\bigtriangleup}
\newcommand{\rest}{\upharpoonright}

\newcommand{\cls}{{\rm{cls\,}}}

\newcommand{\supp}{{\rm{supp\,}}}

\newcommand{\card}{{\rm{card\,}}}

\newcommand{\spann}{{\rm{span}}}

\mathchardef\mhyphen="2D

\swapnumbers
\theoremstyle{plain}
\newtheorem{thm}{Theorem}[section]
\newtheorem{cor}[thm]{Corollary}
\newtheorem{lem}[thm]{Lemma}
\newtheorem{prop}[thm]{Proposition}

\theoremstyle{definition}
\newtheorem{defn}[thm]{Definition}

\newtheorem{rmk}[thm]{Remark}

\newtheorem{prob}[thm]{Problem}

\numberwithin{thm}{section}

\begin{document}

\title
{Hypercyclic operators, Gauss measures and Polish dynamical systems}

\author{Iftah Dayan and Eli Glasner}
\address{Department of Mathematics\\
     Tel Aviv University\\
         Tel Aviv\\
         Israel}
\email{iftahday@math.tau.ac.il}         
\email{glasner@math.tau.ac.il}
 \urladdr{http://www.math.tau.ac.il/$^\sim$glasner}

\date{February 20, 2013}

\begin{abstract}
In this work we consider hypercyclic operators as a special case of
Polish dynamical systems. In the first section we analyze the 
construction of Bayart and Grivaux of a hypercyclic operator
which preserves a Gaussian measure, and derive a description of
the maximal spectral type of the Koopman operator associated to the corresponding
measure preserving dynamical system. We then use this
information to show the existence of a mildly but not strongly mixing hypercyclic 
operator on Hilbert space. In the last two sections we study hypercyclic  
and frequently hypecyclic operators which, as Polish dynamical systems are, M-systems, E-systems, and syndetically transitive systems.
\end{abstract}




\thanks{{This work is supported by ISF grant 
\#1157/08}}

\maketitle
\tableofcontents

A {\em Polish dynamical system} is a pair $(Z,T)$ where $Z$ is a Polish space 
(i.e. a completely metrizable and separable topological space) and
$T : Z \to Z$ is a homeomorphism (or sometimes, more generally, a continuous map
with dense image).

Let $X$ be an infinite dimensional, separable Banach space
(or sometimes, more generally, a separable Fr\'echet space).
We denote by  $\mathcal{L}(X)$ the space of bounded linear
operators on $X$. As usual $X^*$ denotes the dual space of $X$. 
If $T$ is an element of $\mathcal{L}(X)$ with a dense range then we can
consider the system $(X,T)$ as a Polish dynamical system. We then say that
$(X,T)$ is a {\em linear system}. When a linear system is topologically transitive,
i.e. when for any two nonempty open sets $U, V$ in $X$ there is some $n \in \N$
with $T^n U \cap V \not=\emptyset$, then the operator $T$ is called
{\em hypercyclic}.

In this work we will assume some basic knowledge of ergodic theory,
of the general theory of (compact metric) dynamical systems, and of the 
theory of hypercyclic operators. We refer to \cite{F} and \cite{Gl-03} for the first two
and to \cite{BM} for the latter. We thank Benjy Weiss and Sophie Grivaux for some
helpful remarks.

\section{On linear Gauss transformations}


Recall the following definition \cite[Definition 5.7]{BM}.
\begin{defn}
A {\em Gaussian measure on a Banach space $X$} is a probability measure $\mu$ on
$X$ such that each continuous linear functional $x^* \in X^*$ has complex Gaussian distribution, when considered as a random variable on the probability space
$(X,\mathcal{B},\mu)$.
\end{defn}

Following \cite{LPT}, by a
{\em Gaussian probability space} we mean a standard probability space $(Z,\mathcal{B},\mu)$ 
together with an infinite-dimensional closed real subspace $H^r$ of $L_2(Z,\mu)$ such that
the $\sig$-algebra generated by $H^r$, considered as a collection of random variables, is
all of $\mathcal{B}$ and each non-zero function of $H^r$ has a Gaussian distribution.
We refer to the subspace $H = H^r +i H^r$ as the {\em complex Gaussian space} 
(or {\em the first Weiner Chaos}) of $L_2(Z,\mu)$. 
We define a {\em generalized Gaussian automorphism}, or simply a 
{\em Gaussian automorphism}, as an ergodic automorphism $S$ of $(Z,\mathcal{B},\mu)$ 
such that $H$ is invariant under $U_S$, the Koopman operator defined by $S$ on
$L_2(Z,\mu)$.  We call $S$ a {\em standard Gaussian automorphism} 
when $H$ is a cyclic space for $U_S$; i.e., when $U_S \rest H$ has simple spectrum.

There is however another kind of linear space of centered complex Gaussian random variables. 
To make the distinction clear we cite the following two theorems from \cite[Propositions 1.33
and 1.34]{Ja}.

\begin{thm} \label{type1}
If $H$ is a complex linear space of centered complex Gaussian variables, 
then the following are equivalent.
\begin{enumerate}
\item
$H$ is the complexification of some real Gaussian space.
\item
$H=(\Re H)_\C$.
\item
$H = \ol{H}$.
\item
If $\zeta \in H$, then $\ol{\zeta} \in H$.
\item
If $\zeta \in H$, then $\Re  {\zeta}, \Im \zeta \in H$.
\end{enumerate}
\end{thm}

\begin{thm}\label{type2}
If $V$ is a complex linear space of centered complex Gaussian variables, 
then the following are equivalent.
\begin{enumerate}
\item
$V$ is a space of symmetric Gaussian variables (i.e. $\la \zeta$ has the same distribution
as $\zeta$ for every $\zeta \in V$ and $\la \in \C$ with $|\la|=1$).
\item
$V$ and $\ol{V}$ are orthogonal.
\item
$(\Re V)_\C = V \oplus \ol{V}$.
\item
If $\zeta \in V$, then $\Re{\zeta}$ and $\Im\zeta$ are independent.
\item
The real linear mapping $\zeta \mapsto \sqrt{2} \Re\zeta$ is an isometry
of $V$ onto $\Re V$.
\end{enumerate}
\end{thm}

We will refer to spaces of the type described in Theorem \ref{type1}
(\ref{type2}) as {\em Gaussian spaces of the first (second) type}, respectively. 
Of course when $V$ is of the second type then
$H : =V \oplus \ol{V}$ is of the first type.

Suppose now that on the linear system $(X,T)$ there exists a Gaussian $T$-invariant 
measure $\mu$ with $\supp(\mu)\allowbreak =X$. Let $U_T : L_2(X,\mu)
\to L_2(X,\mu)$ denote the Koopman operator of the dynamical system
$\bX = (X,\mathcal{B},\mu,T)$. 
Let $K^* : X^* \to L_2(X,\mu)$ be the conjugate linear map sending an element
$x^* \in X^*$ to the $L_2$-function: $x \mapsto \ol{x^*(x)}
= \ol {\langle x^*, x \rangle}$.
Thus the linear operator $R = KK^*$, is the 
{\em Gaussian covariance operator} corresponding to $\mu$; i.e. the unique operator
$R : X^* \to X$ which satisfies the identity
$$
y^*(R x^*)= 
\int_X \ol {\langle x^*, z \rangle}   \langle y^*, z \rangle \,d\mu(z)
= \langle  x^*, y^* \rangle_{L_2(\mu)}. 
$$
We note that since $U_T$ preserves the subspace of real functions, 
its maximal spectral type $\sig_U$ must be the type 
of a symmetric measure ($\sig$ is said to be {\em symmetric} if $\sig(A) = \sig(\bar{A})$
for any Borel subset $A \subset \T$).

For more details see \cite[Chapter 5]{BM}.

We will consider two dynamical properties of a measure preserving
dynamical system $\Xb = (X,\mathcal{B},\mu,T)$ which may not be familiar to all readers.

\begin{defn}
\begin{enumerate}
\item 
We say that $\Xb$ is {\em rigid} if there is a sequence
$n_k \nearrow \infty$ such that $\lim_{k \to \infty} \mu(T^{n_k} A\tri A) =0$, for every
$A \in \mathcal{B}$. 
\item
We say that $\Xb$ is {\em mildly mixing} if it admits no nontrivial rigid factors. 
An equivalent condition, which we will adopt here, is the spectral condition
$$
\limsup_{|n| \to \infty} |\hat{\sig}_f(n)| < 1,
$$
for every $f \in L_2^0(\mu) = \{f \in L_2(\mu) : \int f \,d\mu=0\}$ with $\|f\| =1$.
Here $\hat{\sig}_f(n)$  is the matrix coefficient $\langle U_T^n f, f \rangle$
of the Koopman operator $U_T$ (see e.g. \cite[Exercise 8.17]{Gl-03}).
\item
Call a probability measure $\rho$ on $\T$ {\em mildly mixing} when
$$
\limsup_{|n| \to \infty} |\hat{\theta}(n)| < 1,
$$
for every probability measure $\theta \ll \rho$ on $\T$.
Then, in these terms, $\Xb$ is mildly mixing iff the maximal spectral type of $U_T$
restricted to $L_2^0(\mu)$
is a mildly mixing measure on $\T$. 
(See \cite{HMP} and page 13 and Proposition III.21 of \cite{Qu}.)
\end{enumerate}
\end{defn}

For more details on rigidity and mild mixing we refer to the original paper of
Furstenberg and Weiss \cite{FW} where the notion of mild mixing was introduced, 
and for further developments to the Notes to Chapter 8 of \cite{Gl-03}. 
In the following theorem parts (2) to (5) are direct corollaries of part (1).
The results in parts (2) and (3) are stated (and given different proofs) e.g. in 
\cite[Proposition 5.36]{BM}, and that of part (4) is in \cite{EG}. 
Parts (1) and (5) seem to be new.

\begin{thm}\label{rho}
Let $V = {\cls}{ K^* X^*}\subset L_2(X,\mu)$. 
Then the closed $U_T$-invariant subspace $V$ is a complex Gaussian space
of the second type.
Denoting $V^r = H^r = \{\Re{f} : f \in V\}$ we have $H = H^r + i H^r = V + \overline{V}$.
Let $\rho$ be the maximal spectral type
of the unitary operator $U = U_T \rest H$ (a symmetric probability measure, 
or rather measure class, on the circle $\T$).
\begin{enumerate}
\item
$H$ is a Gaussian space, and it forms the first Wiener chaos of of $L_2(X,\mu)$.
Thus the automorphism $T$ together with the subspace $H$ define a generalized Gauss 
automorphism. The maximal spectral type of $U_T$ is
$$
\exp(\rho) = \sum_{n=0}^\infty (n!)^{-1}\rho^{*n},
$$ 
where $\rho^{*0} = \delta_1$ and $\rho^{*n} = \rho * \rho * \cdots * \rho$,
($n$ times) is the convolutional $n$-th power, for $n \ge 1$.
\item
The measure dynamical system $\bX = (X,\mathcal{B},\mu,T)$ is ergodic iff
it is weakly mixing, iff the measure $\rho$ is continuous (= atomless).
\item
The system $\bX$ is mixing iff $\rho$ is a Rajchman measure 
(i.e. $\lim_{|n| \to \infty}\hat{\rho}(n) = 0$).
\item
The system $\bX$ is rigid (with respect to the sequence $n_k \nearrow \infty$)
iff $\rho$ is a Dirichlet measure (with $\hat{\rho}(n_k) \to 1$).
\item
The system $\bX$ is mildly mixing iff 
$ \rho$
is mildly mixing 
\end{enumerate}
\end{thm}

\begin{proof}
Since as a collection of functions $X^*$ separates points on $X$ it follows
that the $\sig$-algebra of subsets of $X$ generated by $X^*$ coincides with
the Borel $\sig$-algebra $\mathcal{B}(X)$ (see e.g. \cite[Theorem 2.8.4]{Gl-03}).
Since $\mu$ is a Gaussian measure it follows that the automorphism $T$ of the measure space
$(X,\mathcal{B},\mu)$ together with the closed real subspace $H^r$
of $L_2(X,\mu)$ form a generalized Gauss automorphism of $(X,\mathcal{B},\mu)$.
The proofs of the assertions (2), (3) and (4) are now straightforward, as the spectral properties
in question are shared by $\rho$ and the probability measure 
$\eta = \frac{1}{e-1}(\exp(\rho) - \del_1)$,
which represents the maximal spectral
type of $U_T$ restricted to the subspace $L_2^0(\mu)$.
To see part (5) we recall that the collection $\mathcal{L}_I$ of (complex) measures 
$\mu \in M(\T)$ --- the convolution Banach algebra of complex measures on $\T$ ---
such that the probability measure $\frac{|\mu|}{\|\mu\|}$ is mildly
mixing, forms a closed ideal in $M(\T)$. Thus,
it follows that $\rho$ is mildly mixing iff $\eta = \frac{1}{e-1}(\exp(\rho) - \del_1)$ 
is mildly mixing. (For more details see \cite{HMP}.)
\end{proof}

\begin{rmk}
To the equivalent conditions in part (2) of the above theorem one can add the requirement that the representation on the first Wiener chaos $H$ be weakly mixing (see \cite[Theorem 3.59]{Gl-03}). 
In fact, ergodicity (i.e. the non-existence of nonzero invariant functions) on the first
chaos is also equivalent to weak mixing of $T$.
\end{rmk}


In the next theorem we deal with the situation described in \cite[Lemma 5.35]{BM}:
Let $X$ be a separable Banach space, $T \in \mathcal{L}(X)$, and let $\sig$ be
a probability measure on $\T$. Assume that $T$ admits a finite or
countable family of $\sig$-measurable bounded $\T$-eigenvector fields $(E_i)_{i\in I}$.
For each $i \in I$ let $K_i = K_{E_i} : L_2(\sigma) \to X$ be the operator defined by
$$
K_i(f) = \int_\T f(\la)E_i(\la)\, d\sig(\la).
$$
Let $\mathcal{H} = \oplus_{i \in I}L_2(\sig)$ and let 
$M  : \mathcal{H}  \to \mathcal{H}$, where $M : = \oplus_{i\in I} M_i$ and for each $i$,
$M_i : L_2(\sigma) \to L_2(\sigma)$ is the multiplication operator
$M_i(f)(\la) = \la f(\la)$.  
Let $K = \mathcal{H} \to X$ be defined by
$$
K(\oplus_i f_i) = \sum_{i} \al_iK_i(f_i),
$$ 
where $(\al_i)_{i \in I}$ is a family of positive numbers 
such that $\sum_i \al_i^2\|E_i\|^2_{L_2(\T,\sig,X)} < \infty$.
Finally set $R = K K^*$. (Note that the operators $K$ and $K^*$ here are not
the same as the ones used in Theorem \ref{rho}.)

We further assume that: (i) the family $(E_i)_{i \in I}$ is $\sig$-spanning,
(ii) each operator $K_{E_i} : L_2(\sig) \to X$ is $\gamma$-radonifying  
and (iii) $Ker(K) =0$. 
Then as in \cite[Proposition 5.36]{BM} we conclude
that there exists a $T$ invariant Gaussian probability measure $\mu$ fully supported
on $X$ whose covariance operator is $R$.

If $\sig$ is a continuous probability measure on $\T$ we define the probability 
measure $\check{\sig}$
on $\T$ by the formula $\check{\sig}(A) = \sig(\ol{A})$ for $A  \subset \T$.
The measure $\sig$ is {\em symmetric} when $\sig = \check{\sig}$.
We write $\sig^+ = 2 \sig \ch_{\T^+}$ and $\sig^- = 2\sig \ch_{\T^-}$,
where $\T^{+} = \{z \in \T : \Im{z} \ge 0\}$ and $\T^{-} = \{z \in \T : \Im{z} \le 0\}$.
Thus $\check{\sig^+} = \sig^-$ and $\sig  = \frac12 (\sig^+ + \sig^-)$.
Note that the map $f \mapsto \check{f}$ , where $\check{f}$ is the function $\check{f}(\lambda)=f(\overline{\lambda})$, defines a unitary equivalence between
$(L_2(\T,\sig),M)$ and $(L_2(\T,\check{\sig}),\check{M})$,
where $\check{M} f(\la) = \ol{\la} f(\la)$.

\begin{thm}\label{sig=rho}
Let $X$ be a separable Banach space, $T \in \mathcal{L}(X)$, $\sig$ 
a probability measure on $\T$ and $(E_i)_{i\in I}$ a family satisfying properties (i), (ii)
and (iii) as above. Let $\mu$ be the corresponding Gauusian measure on $X$.
Let $U_T : L_2(X,\mu) \to L_2(X,\mu)$ be the Koopman operator
associated to the measure dynamical system $(X,\mu,T)$.
As in Theorem \ref{rho} set $V = {\cls}{ K^* X^*}$, a closed $U_T$-invariant subspace
of $L_2(X,\mu)$. Then:
\begin{enumerate}
\item
$V$ is a complex Gaussian space of the second type. 
Denote $V^r = H^r = \{\Re{f} : f \in V\}$ and $H = H^r + i H^r = V + \overline{V}$, the corresponding complex Gaussian space of the first type. Then $H$ is the 
first Wiener chaos of $L_2(X,\mu)$. 
\item
The restriction of $U_T$ to $V$ is unitarily isomorphic to the unitary operator  
$M : \mathcal{H} \to \mathcal{H}$, hence
the restriction of $U_T$ to the first Weiner chaos $H =V \oplus \ol{V}$ is unitarily isomorphic to  
$M \oplus M : \mathcal{H}\oplus\check{\mathcal{H}} \to  \mathcal{H}\oplus\check{\mathcal{H}}$,
where $\check{\mathcal{H}} = \oplus_{i \in I}L_2(\check{\sig})$.
In particular the maximal spectral measure $\rho$ coincides with 
the symmetric measure $\rho = \frac12(\sig + \check{\sig})$.
Thus all the assertions of Theorem \ref{rho} hold for $\rho$.
\item
\footnote{We are indebted to B. Weiss for suggesting this trick}
If $\sig$ is symmetric and we apply part {\rm{(2)}} to the measure $\sig^+$ we obtain
the situation where $U_T \rest H$ is unitarily equivalent to
$M : \mathcal{H} \to \mathcal{H}$ with $\mathcal{H} = \oplus_{i \in I}L_2(\sig)$ 
and has maximal spectral type $\rho = \sig$ and multiplicity $\card I$.
\end{enumerate}

\end{thm}

\begin{proof}
Using the notation of \cite[Lemma 5.35]{BM},
let $K_i = K_{E_i} : L_2(\sigma) \to X$ be the operators defined by
$$
K_i(f) = \int_\T f(\la)E_i(\la)\, d\sig(\la).
$$ 
Since we have $TK = KM$, where 
$K = \oplus_{i\in I} K_i : \oplus_{i\in I} L_2(\sig) \to X$,
 it follows by the uniqueness of the covariance operator that 
\begin{align*}
\langle U_T x^* , y^* \rangle_{L_2(\mu)} 
& =
\langle RT^*x^*,y^* \rangle\\
&  = \langle K(K^* T^*)x^*, y^* \rangle \\
&  = \langle  K^*y^*, M^*K^* x^* \rangle_{\mathcal{H}} \\
&  = \langle  M^*K^* x^*, K^*y^* \rangle_{\mathcal{H}^*}.
\end{align*}
Thus, the restriction of $U_T$ to $V = {\cls}{ K^* X^*} \subset L_2(X,\mu)$, is unitarily isomorphic to the unitary operator $M^*$ on $\mathcal{H}^*$, hence also to $M$ on 
$\mathcal{H}$. Now $U_T\rest \ol{V}$ is unitarily equivalent to $M: \check{\mathcal{H}} \to \check{\mathcal{H}}$, and 
we conclude that indeed the restriction of $U_T$ to the first Weiner chaos $H =V \oplus \ol{V}$ is unitarily isomorphic to  
$M \oplus M : \mathcal{H}\oplus\check{\mathcal{H}} \to  \mathcal{H}\oplus\check{\mathcal{H}}$.
Clearly the maximal spectral type of $M \oplus  M$ acting on $\mathcal{H}\oplus\check{\mathcal{H}}$  is represented by the symmetric measure 
$\frac12 (\sig + \check{\sig})$, and we can take $\rho = \frac12 (\sig + \check{\sig})$.
Note that when $\sigma$ is symmetric, $\rho =\sigma = \frac12 (\sig + \check{\sig})$.
Part (3) follows because $L_2(\T,\sig^+) \oplus L_2(\T,\sig^-) 
\cong L_2(\T,\sig)$.
\end{proof}


\begin{rmk}
The condition that each operator $K_{E_i}$ be one-to-one (which is implied by the condition
$Ker(K)=0$) is introduced in order to simplify the formulation of the theorem. 
When it is omitted we need to replace the corresponding
$L_2(\sig)$ by the subspace $Ker(K_{E_i})^{\perp}$ (see the second bulleted remark on 
page 101 of \cite{BM}). This amounts to replacing (in the context of \cite[Lemma 5.35]{BM})
$L_2(\sig)$ by $L_2(\sig_i)$ with $\sig_i \ll \sig$. In fact, by a well known theorem of Wiener
a closed subspace $F \subset L_2(\T,\sig)$ which is invariant under $M$ is of the form
$$
F = \ch_B L_2(\T,\sig) = \{f \in L_2(\T,\sig) : f = 0\  {\text{on}} \ B^c\},
$$
for some Borel subset $B$ of $\T$.
\end{rmk}

\begin{rmk}
We refer to Section 5.6 of \cite{BM} for some natural conditions on the Banach space
$X$ which ensure that the complicated conditions of Theorem \ref{sig=rho} are automatically 
satisfied. For example this is the case when $X$ has type 2 and $T$ admits a perfectly spanning 
set of $\T$-eigenvectors (i.e. $\sig$-spanning with respect to some continuous $\sig$), 
\cite[Theorem 5.38]{BM}. 
And, conversely if $X$ has cotype 2 and it admits a Gaussian invariant measure with full support with respect to which $T$ is weakly mixing, then the $\T$-eigenvectors of $T$ are perfectly spanning  \cite[Theorem 5.46]{BM}. 
Thus, when $X$ is a Hilbert space the condition ``$T$ admits
a perfectly spanning set of $\T$-eigenvectors" is necessary and sufficient 
for $T$ to preserve a fully supported Gaussian measure with respect to which
$T$ is weakly mixing.
\end{rmk}



For the next two corollaries we consider the {\em Kalish construction}, as described e.g.
in \cite[Section 2.2]{Gr-12}. Thus $\sig$ is a continuous probability 
measure on $\T$ and we denote its closed support by $L \subset \T$. 
On the Hilbert space $L_2(\T,d\la)$, where $d\la =  ie^{it}dt$, with $dt$ being Lebesgue measure 
on $[0,2\pi]$,
we define the invertible operator $T = M - J$, where
$Mf(\zeta) = \zeta f(\zeta)$ and $Jf(\zeta) = \int_{(1,\zeta)} f(\la) \,d\la$.
Let $\chi_\zeta$ denote the characteristic function of the arc $(\zeta,1)$. 
Then for each $\la \in \T$ the function $\chi_\la$ satisfies $T \chi_\la = \la \chi_\la$.
Let $H_L$ be the closed subspace of $L_2(\T,d\la)$ spanned by the 
collection $\{\chi_\la : \la \in L\}$. Let $E : L \to H_L$ be defined by
$E(\la) = \chi_\la$, and let $K : L_2(\sig) \to H_L$ be the map
$$
f \mapsto \int_L f(\la)E(\la)\, d\sig(\la).
$$ 


\begin{lem}\label{iii}
The $\T$-eigenvector field $E$ and the corresponding operator $K$ satisfy
the properties (i), (ii) and (iii) (see the paragraph preceding Theorem \ref{sig=rho}).
\end{lem}

\begin{proof}
Since $H_L$ is a Hilbert space, and since $E$ is a continuous $\T$-eigenvector field 
we only need to prove (iii); namely that $Ker(K) =0$. 
Suppose $K(f) = 0$ for some $f \in L_2(\sig)$; i.e. the function
$\int_L f(\la)E(\la)\, d\sig(\la)$ is the zero function in $H_L$. This implies that
for every $x^* \in H^*_L$ the numerical integral $\int_L f(\la)x^*(E(\la))\, d\sig(\la) = 0$.
The general $x^* \in H^*_L$ has the form $x^*_g$, 
where $g \in L_2(\T,d\la)$ and for $h \in H^*_L$,
$x^*_g(h) = \int_\T \ol{g(\zeta)}h(\zeta)\, d\zeta$.
Thus we have 
$$
\int_L f(\la) \left(\int_\T \ol{g(\zeta)} \chi_\la(\zeta))\, d\zeta \right)\, d\sig(\la) = 0,
$$
for every $g \in L_2(\T,d\la)$.
This clearly implies that $f =0 $ in $L_2(\sig)$ and our proof is complete.
\end{proof}

\begin{cor}\label{mild}
There exist a separable Hilbert space $X$,
$T \in \mathcal{L}(X)$, and a $T$-invariant
Gaussian measure $\mu$ on $X$ with full support such that the corresponding
measure preserving system $\Xb = (X,\mathcal{B},\mu,T)$ is mildly but not strongly mixing.
\end{cor}

\begin{proof}
Start with a continuous symmetric probability measure $\sig$ on $\T$ which is 
mildly mixing but not Rajchman. (For the existence of such measures
see e.g. \cite{HMP}. Alternatively one can start with any measure theoretically 
mildly but not strongly mixing dynamical system $(Y,\mathcal{B},\nu,S)$ and then let 
$\sig = \sig_{U_S}$.)
Follow the {\em Kalish construction}, with the measure $\sig$,
to obtain a linear operator $T$ in $\mathcal{L}(X)$,
on a separable Hilbert space $X$, and a $T$-invariant fully supported Gaussian probability 
measure $\mu$ on $X$. By Lemma \ref{iii} Theorem \ref{sig=rho} (2) applies
and we know that the maximal spectral type of $U_T \rest H$ 
(with $H^r =  V^r,\  V={\cls}{K^* X^*}$) is $\sig=\frac12 (\sig + \check{\sig})$. 
Thus, by Theorem \ref{rho}, the maximal spectral type of $U_T$
is $\exp(\sig)$. By parts (3) and (5) of this theorem we conclude that the system
$\bX   = (X,\mathcal{B},\mu,T)$ is indeed mildly but not strongly mixing.
\end{proof}

In the general construction described in Theorem \ref{rho} there is, a priori, no
reason for the Gauss automorphism $T$ to be standard. 
We next address the question: 
when is the space $H$ cyclic for $U$ ?

\begin{cor}
Given a symmetric, continuous probability measure $\sig$ on $\T$,
there exist a Hilbert space $X$,  $T \in \mathcal{L}(X)$ and  
a $T$-invariant Gaussian measure $\mu$ on $X$ with full support, such that the corresponding 
unitary operator $U$ on $H$ is cyclic and has simple spectrum with maximal spectral type 
$\sig$. 
\end{cor}

\begin{proof}
Apply the Kalish construction
to obtain a Hilbert space $X$, a hypercyclic operator
$T\in \mathcal{L}(X)$, and a fully supported, $T$-invariant, Gaussian, probability measure 
$\mu$ on $X$. 
Let $E : \T \to X$ be the Kalish continuous and spanning $\T$-eigenvector field,
and let $K = K_E : L_2(\sigma) \to X$ be the corresponding operator, defined by
$$
K(f) = \int_\T f(\la)E(\la)\, d\sig(\la).
$$ 
Finally with $\mathcal{H} = L_2(\T,\sig)$ let $M : \mathcal{H}  \to \mathcal{H} $ be the 
multiplication operator $M(f)(\la) = \la f(\la)$. Of course the operator $M$ has simple spectrum
with maximal spectral type $\sig$. 
An application of Theorem \ref{sig=rho} (3) completes the proof.

\end{proof}

\begin{rmk}
It is perhaps not unreasonable to surmise that such a ``linear model"
for a standard Gaussian stochastic process, where the probability space
is a Hilbert space, the transformation is a linear map, and the 
random variables are linear functionals, may become
useful in other fields where Gaussian processes are being used.
\end{rmk}

Relying on an intricate construction of Eisner and Grivaux \cite{EG}
one can deduce the existence of a hypercyclic operator on a separable Hilbert space which
is weakly but not mildly mixing, both in the measure theoretical and the
topological sense. (For the definition of topological mild mixing see e.g.
\cite{GW2}.)

\begin{prop}
\begin{enumerate}
\item
(\cite[Theorem 4.1]{EG}) There exists a hypercyclic operator $T$ on a Hilbert space $H$
and a fully supported $T$-invariant Gaussian measure $\mu$ such that the
corresponding measure preserving system is weakly mixing and rigid, hence not
mildly mixing.
\item
(\cite[Theorem 1.13]{EG})
There exists a hypercyclic operator $T$ on a Hilbert space $H$
and a fully supported $T$-invariant Gaussian measure $\mu$ such that the
corresponding measure preserving system is weakly mixing, and such
that for some sequence $n_k \nearrow \infty$ we have $\| T^{n_k} - I \| \to 0$.
These facts show that the Polish dynamical system $(H,T)$ is topologically
weakly but not mildly mixing.
\end{enumerate}
\end{prop}

\begin{proof}
The last sentence is the only claim which needs a proof. This is 
a straightforward analogue of the proof of Lemma 1.14 in \cite{GW2} and we 
leave it to the reader. 
\end{proof}

\section{Upper and lower frequently hypercyclic operators}

In \cite[Proposition 6.23]{BM} it is shown that if there is 
on $X$ a $T$-invariant probability measure $\mu$ with full support 
with respect to which $T$ is ergodic, then $T$ is frequently hypercyclic.

Motivated by the theorem below we propose the following definitions.

\begin{defn}
Let $(X,T)$ be a Polish dynamical system. We say that it is {\em upper-frequently
transitive} (UFT for short) if there is a point $x_0 \in X$ such that for every nonempty
open subset $U \subset X$ we have
$$
ud (N(x_0,U)) = \limsup_{N \to \infty} \frac
{\card(N(x_0,U) \cap [1,N])}{N} > 0.
$$  
Similarly, using 
$$
ld (N(x_0,U)) = \liminf_{N \to \infty} \frac
{\card(N(x_0,U) \cap [1,N])}{N} > 0.
$$ 
instead of $\limsup$ in the above definition, we obtain the notion
of a {\em lower-frequently transitive} (LFT for short) Polish dynamical system.
Of course LFT implies UFT.
\end{defn}

\begin{defn}
Let $(X,T)$ be a Polish dynamical system. We say that a compact metric dynamical system
$(\tilde{X},\tilde{T})$ is a {\em dynamical compactification} of $(X,T)$ if 
there is an equivariant homeomorphism $J : X \to \tilde{X}$ 
(i.e. $\tilde{T} \circ J = J \circ T$ on $X$) with dense image.
It then follows that $J(X)$ is a dense $G_\delta$ subset of $\tilde{X}$.
Let us recall that to every dynamical compactification $J : X \to \tilde{X}$ corresponds
a unique separable unital $C^*$-algebra $\mathcal{A} \subset C_b(X,\C)$ such that
the map $J^* : \mathcal{A} \to C(\tilde{X},\C)$ defined by
$$
J^*f(x) = \tilde{f}(Jx) = f(x), \quad {\text{for every $f \in \mathcal{A}$ and every $x \in X$}},
$$
is a surjective isomorphism of the corresponding $C^*$-algebras.
Conversely, with every separable unital $C^*$-subalgebra $\mathcal{A}$ of $C_b(X,\C)$ 
which separates points and closed subsets of $X$ there is an associated
dynamical compactification $J : X \to \tilde{X}$, where $\tilde{X}$ is the compact metrizable  Gelfand space which corresponds to $\mathcal{A}$.
In the sequel we will sometimes suppress the map $J$ and consider $X$ as a subset of $\tilde{X}$.
\end{defn}

\begin{defn}
Let $(X,T)$ be a compact metric dynamical system and $\mu$ a probability measure
on $X$.  A point $x_0 \in X$ is a {\em generic point} for $\mu$ if 
$\lim_{n \to \infty}\mu_n = \mu$ in the weak$^*$ topology,  where
$$
\mu_n = \frac{1}{n} \sum_{j=1}^n \delta_{T^j x_0}.
$$
The point $x_0$ is {\em quasi-generic} for $\mu$ if for some subsequence 
$\lim_{k \to \infty}\mu_{n_k} = \mu$ in the weak$^*$ topology.
Clearly a measure which admits a quasi-generic point is necessarily $T$-invariant.
However, it need not be ergodic even when it admits a generic point.
\end{defn}

\begin{thm}\label{quasigeneric}
Let $(X,T)$ be a Polish dynamical system.  
\begin{enumerate}
\item
If $T$ is lower-frequently transitive
then for every metric compact dynamical compactification 
there exists on $\tilde{X}$ a $\tilde{T}$-invariant probability measure $\mu$ with 
$\supp(\mu)=\tilde{X}$, 
which moreover admits a quasi-generic point in $X$.
\item
Conversely, if $(X,T)$ admits a dynamical compactification $(\tilde{X},\tilde{T})$
such that on $\tilde{X}$ there is a $\tilde{T}$-invariant probability
measure $\mu$ with $\supp(\mu)=\tilde{X}$, which moreover admits a 
quasi-generic point in $X$, then $(X,T)$ is upper-frequently transitive.
\end{enumerate} 
\end{thm}

\begin{proof}


Suppose first that $T$ is lower-frequently transitive and fix a point $x_0 \in LFT(X)$.
Let $J : X \to \tilde{X}$ be a given dynamical compactification.
Let $\{U_i\}_{i \in \N}$ be an enumeration of a basis, consisting of balls, for the topology on $X$.
For each $n \in \N$ let 
$$
\mu_n = \frac{1}{n} \sum_{j=1}^n \delta_{T^j x_0},
$$
a probability measure on $X$. Via a diagonal process we can define a subsequence
$\{\mu_{n_k}\}$ with the property that for every $i$ there is $k_i$ such that for every
$k > k_ i$ we have $\mu_{n_k}(U_i) >\frac12 ld(N(x_0,U_i))$, where
$$
N(x_0,U_i) = \{t \in \N : T^t x_0 \in U_i\},
$$
and
$$
ld (N(x_0,U_i)) = \liminf_{N \to \infty} \frac
{\card(N(x_0,U_i) \cap [1,N])}{N}.
$$ 

Next, for each $i$, let $B_i \subset U_i$ be a slightly smaller open ball and choose
a continuous function $f_i \in \mathcal{A}$, $f_i : X \to [0,1]$,
where $\mathcal{A}$ is the unital $C^*$-subalgebra of $C_b(X,\C)$ which corresponds 
to $J$, such that $f_i \rest B_i \equiv  1$  and 
$f_i$ vanishes on the complement of $U_i$. 
Now, $\tilde{X}$ being compact and metric, the sequence of probability measures $\{\mu_{n_k}\}$ 
has a convergent subsequence, which for brevity we still denote by  $\{\mu_{n_k}\}$,
say $\lim \mu_{n_k} = \mu$. Clearly $\mu$ is a $\tilde{T}$-invariant probability measure
on $\tilde{X}$, and, by construction, 
$$
\int_{\tilde{X}} \tilde{f}_i \,d\mu > 0,
$$
for every $i \in \N$. This implies that $\mu$ has full support on $\tilde{X}$, and by our construction
the point $x_0$ is a $\tilde{T}$-quasi-generic point for $\mu$.

Now assume that the condition (2) in the theorem is satisfied and let $x_0 \in X$ be such that
$z_0 = J(x_0) \in \tilde{X}$ is a 
$\tilde{T}$-quasi-generic point for $\mu$ with respect to a sequence $n_k \nearrow \infty$, 
then for every $i$ we have
\begin{align}
\begin{split}
ud(N(U_i,x_0)) & \ge \limsup_{N \to \infty} \frac1N
\sum_{j =1}^N f_i(T^j x_0)\\
& = 
\limsup_{N \to \infty} \frac1N
\sum_{j =1}^N  \tilde{f}_i(J(T^j x_0))\\
& \ge
\lim_{k \to \infty} \frac1{n_k}
\sum_{j =1}^{n_k}  \tilde{f}_i(\tilde{T}^j z_0)
= \int_{\tilde{X}} \tilde{f}_i \,d\mu > 0.
 \end{split}
\end{align}




\end{proof}

\begin{rmk}
In the case where the system $(X,T)$ is a hypercyclic system on a Fr\'echet space $X$,
we note that our notion of a LFT system coincides with the well known notion of a frequently hypercyclic linear system.
In order to avoid confusion we refer to a hypercyclic system which, as a Polish system,
is UFT (LFT) as  an {\em upper-frequently hypercyclic (lower-frequently hypercyclic)},
or UFH (LFH) system respectively. 
\end{rmk}

\section{Three residual properties}

The works \cite{W} and \cite{AG} both aim at developing a dynamical theory of Polish
systems. In \cite{W} Weiss is dealing directly with homeomorphisms of Polish spaces.
He refers to this theory as {\em generic dynamics} and calls properties of such systems 
{\em generic}. 
In \cite{AG} our approach is via compact systems and the main tool we use is that
of residual properties.

A property P of compact metric dynamical systems is called a {\em  residual property}
if it satisfies the following three conditions:
\begin{itemize}
\item
P is preserved under factors.
\item
P is preserved under inverse limits, and
\item
P lifts through almost one-to-one extensions.
\end{itemize}

The last property is perhaps the most important of the three. Recall that a factor map
$\pi : (X,T) \to (Y,S)$ of compact metric dynamical systems is {\em almost one-to-one}
if it satisfies one (and hence both) of the following equivalent properties:
\begin{enumerate}
\item
The set 
$$
X_0 = \{ x \in X : \pi^{-1}(\pi(x)) = \{x\}\}
$$ 
is a dense $G_\delta$ subset of $X$.
\item
For any closed subset $Z \subset X$ the condition $\pi(Z) = Y$
implies $Z = X$.
\end{enumerate}
The latter property is called {\em irreducibility} of the map $\pi$.

\begin{defn}
We say that two compact metric systems $(X,T)$ and $(Y,S)$ are {\em residually
isomorphic} if there are invariant dense $G_\del$ subsets $X_0 \subset X$ and $Y_0
\subset Y$ and an isomorphism of Polish systems $\phi :  (X_0,T) \to (Y_0,S)$.
\end{defn}

It is now easy to check that a residual property of compact metric systems is an invariant under residual isomorphisms, and it follows that we can actually regard residual properties as 
properties of Polish systems.
In \cite{AG} many familiar properties of compact metric systems are shown to be
residual. Among others we have in this list the following properties: minimality,
weak mixing, M-systems, E-systems, $\mathcal{F}$-transitivity for any proper family 
$\mathcal{F}$ of subsets of $\Z$, weak disjointness from a given residual property, 
and many more. 
We will be mostly interested in the following three residual properties:

\begin{defn}
\begin{enumerate}
\item
A Polish dynamical system $(X,T)$ such that the sets of the form $N(U,V)$, 
where $U,V$ are nonempty open subsets of $X$, are all syndetic is called 
{\em syndetically transitive} (or sometimes {\em topologically ergodic}).
\item
A Polish dynamical system is an {\em M-system} if it is topologically transitive and
has the property that in some (hence any) dynamical compactification $J : (X,T) \to (\tilde{X}, \tilde{T})$ 
the compact system $(\tilde{X}, \tilde{T})$ is an M-system; i.e. the union of the minimal subsets of
$\tilde{X}$ is dense in $\tilde{X}$. 
\item
A Polish dynamical system is an {\em E-system} if it is topologically transitive and
has the property that in some (hence any) dynamical compactification 
$J : (X,T) \to (\tilde{X}, \tilde{T})$ the compact system $(\tilde{X}, \tilde{T})$ is an E-system; 
i.e. the union of the supports of ergodic measures on $\tilde{X}$ is dense in $\tilde{X}$ 
(equivalently, if there is a $\tilde{T}$-invariant measure on $\tilde{X}$ with full support; 
see \cite{GW1}).  
\end{enumerate}
\end{defn}

Recall that a subset $S$ of $\N$ or $\Z$ is called {\em syndetic} if the
gaps in $S$ are uniformly bounded. A subset $L \subset \N$ (or $\Z$) is called {\em thick} if it contains arbitrarily large intervals. The families $\mathcal{S}$ and $\mathcal{L}$
of syndetic and of thick sets respectively are {\em dual families}; i.e.
$S \in \mathcal{S}$ iff $S \cap L \not= \emptyset$ for every $L \in \mathcal{L}$,
and $L \in \mathcal{L}$ iff $L \cap S \not= \emptyset$ for every $S \in \mathcal{S}$.

In order to demonstrate the notion of residual property let us show that the property of
being an E-system is indeed residual.


\begin{lem}
Being an E-system is a residual property.
\end{lem}

\begin{proof}
Clearly both topological transitivity and the existence of a fully supported invariant measure
are properties which are preserved under factors and inverse limits. 
Also it is well known and easy to see that
topological transitivity lifts through an almost one-to-one extension. 
Finally, suppose
$\pi : (X,T) \to (Y,S)$ is an almost one-to-one extension, where $(Y,S)$ is an E-system.
Let $\nu$ be an $S$-invariant probability measure on $Y$ with full support.
By compactness there is some probability measure on $X$ whose push-forward
under $\pi$ is $\nu$, so that the set $Q$ of all probability measures $\eta$ on $X$ with
$\pi_*(\eta) =\nu$ is a nonempty convex weak$^*$ compact subset of $C(X)^*$.
Clearly $Q$ is also $T$-invariant  and by the Markov-Kakutani fixed point theorem
there is a measure $\mu \in Q$ which is $T$-invariant. Let $Z = \supp(\mu)$
then, as $\pi_*(\mu)=\nu$, we have $\pi(Z) = X$ and by the irreducibility of $\pi$
we conclude that $Z = X$.
\end{proof}

The reader should be warned that the property of {\em Devaney chaos} or, in the terminology
of \cite{GW1}, of being a {\em P-system} --- namely
topological transitivity plus the requirement that the periodic points be dense ---
is not a residual property; it is preserved by neither inverse limits nor
almost one-to-one extensions (see \cite{AG}).

As a direct corollary of Theorem \ref{quasigeneric} we have:

\begin{cor}
Every frequently hypercyclic system is an E-system.
\end{cor}

For more details on M-systems, E-systems and topologically ergodic systems see \cite{GW1}
and \cite{AG}.

\begin{lem}\label{UV}
Let $(X,T)$ be a topologically transitive Polish dynamical system and $U, V \subset X$
nonempty open sets.
\begin{enumerate}
\item
There exists a nonempty open subset $W$ and $n \in \N$ such that
$$
n + N(W,W) \subset N(U,V).
$$
\item
For any $x_0$, a transitive point of $X$ we have:
$$
N(U,U) = N(x_0,U) - N(x_0,U).
$$
\end{enumerate}
\end{lem}

\begin{proof}
(1)\ 
By topological transitivity, given $U,V$ nonempty open sets, there exists $n \in \N$
with $W : = U \cap T^{-n}V \ne\emptyset$.  For  $k \in N(W,W)$ there is $x \in W$ such that
$T^kx \in W$, hence $T^{n + k}x \in V$ , and we conclude that $n + N(W,W) \subset N(U,V)$. 

(2)\
Given $n \in N(U,U)$ there is some $k \in N(x_0,U)$ such that $T^n T^k x_0 \in U$.
Thus $n = (n+k) - k \in N(x_0,U) - N(x_0,U)$. Conversely, if $T^k x_0 \in U$ and
$T^l x_0 \in U$ then $T^{k - l}T^lx_0 = T^k x_0 \in U$, hence $k - l \in N(U,U)$.
\end{proof}

\begin{defn}(\cite[Definition 3.7]{F})
Let $L$ be a subset of either $\N$ or $\Z$. The {\em upper Banach density} of $L$ is
$$
ubd(L) = \limsup_{|I| \to \infty} \frac{|L \cap I |}{|I|},
$$ 
where $I$ ranges over intervals of $\N$ or $\Z$. 
\end{defn}

For the proof of the next lemma see e.g. \cite[Proposition 3.19]{F}.
\begin{lem}\label{UBD}
If $L \subset \Z$ has positive upper Banach density, then $L - L$ is syndetic.
\end{lem}

The next result is from \cite{GW3}.

\begin{thm}\label{M-synd}
For Polish dynamical systems:
\begin{enumerate}
\item
Every M-system is an E-system.  
\item
Every E-system is syndetically transitive.
\end{enumerate}
\end{thm}

\begin{proof}
Since all these properties are residual 
we can and will assume that the systems in question are compact metric.

(1)\ 
Let $(X,T)$ be a compact metric M-system.
By the Krylov-Bogolubov theorem (see e.g. \cite[Theorem 4.1]{Gl-03}) every minimal 
set carries an invariant measure (which is necessarily
fully supported). Now let $\{M_i\}$ be a (countable) collection of minimal subsets of $X$
whose union is dense in $X$. For each $i$ let $\mu_i$ be an invariant probability measure
supported on $M_i$, and set $\mu = \sum_{i =1}^\infty 2^{-i}\mu_i$. Then $\mu$ is an invariant measure of full support, so that $(X,T)$ is an E-system.

(2)\ 
Let $(X,T)$ be a compact metric E-system.
Let $U$ be a nonempty open subset of $X$ and choose a point $x_0 \in X$ which is generic
for an ergodic measure $\mu$ with $\mu(U) >0$.  
We have  $N(U,U) \supset  N(x_0,U) - N(x_0,U)$. Now, since the set $N(x_0,U)$ has positive
density it follows that $N(U,U)$ is syndetic (Lemma \ref{UBD}). 
Finally, since in a topologically transitive system every set of the form $N(U,V)$ contains a
translate of a set of the form $N(W,W)$ (Lemma \ref{UV}(2)), our assertion follows.
\end{proof}

\begin{rmk}
Combining claims (1) and (2) of the above theorem we deduce that an M-system
is syndetically transitive. Here is an easy direct proof of this result.

\begin{proof}
Again we can assume that $X$ is compact.
By Lemma \ref{UV}(1) it suffices to show that sets of the form $N(U,U)$ are syndetic. Since by assumption there is a minimal point $x \in U$, and as $N(x,U) \subset N(U,U)$, the 
Gottschalk-Hedlund criterion (see e.g. \cite[Exercise 1.1.2]{Gl-03}) shows that $N(U,U)$ is indeed syndetic.
\end{proof}
\end{rmk}

\begin{defn}
A linear hypercyclic system $(X,T)$ which is syndetically transitive will be 
called {\em syndetically hypercyclic}. In view of the theorem above every chaotic 
linear system is syndetically hypercyclic. (We should warn the reader that
this notion is not the same as the one defined in \cite[Corollary 4.7]{BM}.)
\end{defn}

Lemma 4.9 in \cite{BM} says that in a hypercyclic system $(X,T)$, sets of the form
$N(U,W)$ and $N(W,V)$,  where $U, V$ are nonempty open subsets of $X$ 
and $W$ is an open neighborhood of $0 \in X$, are always thick. 
In fact, the same
proof shows that this is true for any Polish transitive system which is topologically transitive 
and has a fixed point (where we now assume that $W$ is a neighborhood of the fixed point).  

\begin{thm}\label{synd-wm}
A syndetically hypercyclic system  is weakly mixing.
\end{thm}

\begin{proof}
We obtain a short proof by applying the ``three open sets" condition
(\cite[Theorem 4.10]{BM}).  Given $U,V$ nonempty open subsets of $X$ 
and an open neighborhood $W$ of $0 \in X$, we observe that 
the set $N(U,W)$ is thick by the above lemma, while the set $N(W,V)$ is
syndetic by assumption. Hence $N(U,W) \cap N(W,V) \ne \emptyset$.
\end{proof}

\begin{thm}\label{UFH-synd}

If $(X,T)$ is UFT then it is syndetically transitive.
In particular a frequently hypercyclic system is syndetically hypercyclic.
\end{thm}

\begin{proof}
If $U$ is any nonempty open set, we can choose a point $x_0 \in U \cap UFT(X)$
and then $N(U,U) =  N(x_0,U) - N(x_0,U)$ (Lemma \ref{UV}(2)). 
By assumption the set $N(x_0,U)$ has
positive upper density hence,  as before, $N(U,U)$ is syndetic. 
Since in a topologically transitive system every set of the form $N(U,V)$ contains a
translate of a set of the form $N(W,W)$ (Lemma \ref{UV}(1)), our assertion follows.
\end{proof}

\begin{rmk}
Combining the statements of the last two theorems we retrieve the result of 
Grosse-Erdmann and Peris \cite{G-EP} 
which asserts that a frequently hypecyclic operator is weakly mixing. 
Basically our proof is the same as theirs.
Also theorems \ref{synd-wm}, \ref{UFH-synd} and Proposition \ref{eigen-M} below were 
already obtained in \cite{BG}
\end{rmk}

\begin{prop}\label{eigen-M}
If $(X,T)$ is a hypercyclic system such that:
\begin{quote}
{\em there is a subset $D$ of $X$ consisting of $\T$-eigenvectors whose span is dense in $X$},
\end{quote}
then $T$ is an M-system, hence syndetically hypercyclic.
\end{prop}

\begin{proof} 
For every $x \in \spann\; D$ the orbit closure $\ol{\mathcal{O}_T(x)}$ is a
minimal rotation on a finite dimensional torus. Thus under the condition of the theorem
$(X,T)$ is an M-system, hence also syndetically transitive by Theorem \ref{M-synd}.
\end{proof}

\begin{prop}
If $(X,T)$ is a hypercyclic system such that:

(i) $X$ is of cotype 2,

(ii) there is a  $T$-invariant, nondegenerate, Gaussian measure on $X$,

then $(X,T)$ is an M-system, hence syndetically hypercyclic.
\end{prop}

\begin{proof} 
By \cite[Theorem, 4.1]{Ba-Gr} these conditions imply the existence of a set of
$\T$-eigenvectors whose linear span is dense in $X$. Now apply Proposition \ref{eigen-M}.
\end{proof}

\begin{prop}
There are syndetically hypercyclic systems on Hilbert space which are not chaotic.
\end{prop}

\begin{proof}
The proof of Theorem 6.41 in \cite{BM}, which asserts the existence of a frequently hypercyclic operator 
in $\mathcal{L}(H)$ which is not chaotic, applies here as well, almost verbatim.  We only
have to use Proposition \ref{eigen-M} instead of \cite[Lemma 6.38.1]{BM}.
\end{proof}


The following diagram may help the reader to sort out the main results of the last two sections.
For a hypercyclic operator $T$ on a Banach space $X$ we have the following
implications:

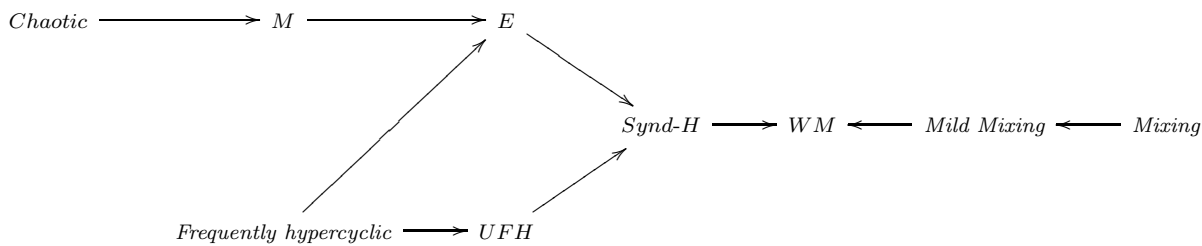
\begin{equation*}
\tiny{
\xymatrix
{
Chaotic \ar[r] &  M \ar[r] & E \ar[dr] &  & & &\\
 &  & &  Synd{\mhyphen}H  \ar[r]& WM  & {\text{\it{Mild Mixing}}} \ar[l]  & {\text{\it{Mixing}}} \ar[l]\\
& {\text{\it{Frequently hypercyclic}}}\ar[uur]  \ar[r] & UFH  \ar[ur] & & & &
}}
\end{equation*}

\begin{prob}
Provide examples to show that none of these implications can be reversed.
(Of course some cases here are already known.)
\end{prob}

%
\br
 
\bibliographystyle{amsplain}

\end{document}